\documentclass[11pt, oneside]{amsart}   	
\usepackage{geometry}                		
\usepackage{graphicx}				
\usepackage{comment}
\usepackage{amsmath}								
\usepackage{amssymb,verbatim}
\usepackage{color}

\newtheorem{thm}{Theorem}[section]
\newtheorem{prop}[thm]{Proposition}

\newtheorem{que}{Question}
\newtheorem{defn}[thm]{Definition}
\newtheorem{lem}[thm]{Lemma}

\newtheorem{conj}[thm]{Conjecture}
\newtheorem*{rk}{Remark}

\newcommand{\del}{\partial} 
\newcommand{\dbar}{\overline{\del}}
\newcommand{\ddb}{\sqrt{-1}\del\dbar}
\newcommand{\ddbar}{\del\dbar}
\newcommand{\ddt}{\frac{d}{d t}}

\newcommand{\fa}{\mathfrak{a}}
\newcommand{\cP}{\mathcal{P}}

\title{The deformed Hermitian-Yang-Mills equation in geometry and physics}
\author{Tristan C. Collins$^{*}$}
\address{Department of Mathematics\\
 Harvard University\\
 1 Oxford St.\\
 Cambridge\\
 MA 02138\\
 USA}
 \email{tcollins@math.harvard.edu}
\author{Dan Xie}
\email{dxie@cmsa.fas.harvard.edu}
\author{Shing-Tung Yau}
\email{yau@math.harvard.edu}

\dedicatory{To Nigel Hitchin, with admiration, on the occasion of his 70th birthday.}
\thanks{$^{*}$ supported in part by NSF grant DMS-1506652}
\date{}							

\begin{document}
\maketitle

\begin{abstract}
We provide an introduction to the mathematics and physics of the deformed Hermitian-Yang-Mills equation, a fully nonlinear geometric PDE on K\"ahler manifolds which plays an important role in mirror symmetry.  We discuss the physical origin of the equation, discuss some recent progress towards its solution.  In dimension $3$ we prove a new Chern number inequality and discuss the relationship with algebraic stability conditions.
\end{abstract}

\section{The deformed Hermitian-Yang-Mills equation and Mirror Symmetry}\label{sec: phys}

It was discovered a long time ago that there are five perturbatively well-defined 10 dimensional superstring theories: Type IIA, type IIB, type I, heterotic SO(32) and heterotic $E_8\times E_8$ string theory; see \cite{Beck} for an introduction into string theory. 
 To get a realistic particle physics model in four dimensions, one needs to study compactifications of string theory
on compact 6 (real) dimensional manifolds. It turns out that Calabi-Yau three manifolds \cite{CHSW} play a crucial role in studying supersymmetric compactifications, for which many computations are under control.
 By choosing different Calabi-Yau geometries and different 
10 dimensional string theories, one gets a huge number of string theory vacua in various dimensions. These theories were originally thought to be independent with no obvious relations between them.  
One of the main discoveries of the second string revolution in the mid-nineties was that these string vacua are, in fact, not independent at all, and most of them are related through various kinds of string dualities.  A fundamental example of this was the discovery of pairs of Calabi-Yau manifolds $X, \hat{X}$ for which IIA string theory on $X$ (resp. $\hat{X}$) is equivalent to IIB string theory on $\hat{X}$ (resp. $X$) \cite{GP}.  This duality came to be called mirror symmetry.  Mirror symmetry has generated a huge amount of interest among physicists and mathematicians, in part due to its surprisingly successful prediction of enumerative curve counts inside Calabi-Yau manifolds \cite{COGP}.   One basic feature of mirror symmetry is that it exchanges the complexified Kahler moduli space and complex structure moduli space of the mirror pairs. Mirror symmetry often maps one hard quantum problem to a simpler 
classical geometric problem, for example, the very complicated counting of curves on $X$ is reduced to simpler computation of period integrals on $\hat{X}$. 

A second major discovery in the second string revolution was the existence of various kinds of extended objects, besides the fundamental string, which are used to define string theory.  One of most important classes of these extended objects 
is the class of D-branes.  These new discoveries provided new insights into the understanding of mirror symmetry. Using T-duality and D-branes, Strominger, Yau and Zaslow  described a geometric picture of mirror symmetry which is now called the SYZ picture \cite{SYZ}. 

The first ingredient of the SYZ picture of mirror symmetry is the so-called T-duality symmetry of string theory \cite{Beck}.  T-duality relates different string theories compactified on circles.
Consider a string theory ${\mathcal T}$ compactified on a circle $S_A$ with radius $R$. T-duality predicts that it should be equivalent to a different string theory ${\mathcal T}^{'}$ 
compactified on a circle $S_B$ with radius ${1\over R}$. The typical example is type IIA string theory and type IIB string theory which, when compactified on a circle, are related by T-duality.

The second ingredient of the SYZ picture of mirror symmetry is D-branes. The name D-brane is derived from the fact that the world sheet string theory description of these objects has Dirichlet (``D") boundary conditions on the world volume of the brane. We often denote 
a D-brane as a D$p$ brane where $p$ denotes the number of spatial dimensions of the brane world volume, and the full space-time dimension of a D$p$ brane is $p+1$.  D-branes are extended objects carrying Ramond-Ramond (RR) charges.  Not all D-branes are physically realistic.  The realistic D-branes are minimizers of some energy functional, and are usually referred to as BPS. 
The BPS branes of type II string theory on a Calabi-Yau manifold $X$ have been classified into two kinds in the large volume/ large complex structure limit: one type of D-brane is a special Lagrangian submanifold of $X$ and the other type 
is a complex submanifold. A D-brane has a quantum moduli space which is related to the geometry it probes in an interesting way; for example, the quantum moduli space of D0 brane probing a Calabi-Yau manifold 
$X$ is nothing but $X$ itself. 

Let us consider how T-duality acts on D-branes;  T-duality maps a D-brane wrapping on circle $S_A$ to a D-brane sitting on a point of the dual circle $S_B$ (and vice versa).  Combining T-duality and the existence of D-branes, SYZ proposed the following geometric picture of mirror symmetry: consider a 3-dimensional Calabi-Yau manifold X which has a
$T_3$ fibration. If we compactify type IIB string theory on $X$ and apply T-duality to every fiber, we should get a type IIA string theory on the mirror manifold $\hat{X}$.  BPS D-branes of the type IIB string are required to be special Lagrangians, while BPS D-branes of the type IIA string are required to be complex submanifolds.  If we wrap a D3-brane on a T3 fiber, we get D0-brane on the mirror manifold after T-duality.  The quantum moduli space of original D3-brane should equal to the moduli space of D0-brane which is then equal to the mirror manifold $\hat{X}$.  Therefore $\hat{X}$ arises as the dual torus fibration and its geometry can be understood from the D-brane moduli space associated special Lagrangian torus fibre of $X$. 

\begin{conj}[Stominger-Yau-Zaslow, \cite{SYZ}] \label{conj: SYZ}
Let X and $\hat{X}$ be a mirror pair of CY manifolds.  Near the large volume/ large complex structure limits:
\begin{itemize}
\item $X$ and $\hat{X}$ admit dual special Lagrangian torus fibrations $\mu: X \rightarrow B$ and $\hat{\mu}: \hat{X} \rightarrow B$ over the same base $B$.
\item  There exists a fiberwise Fourier-Mukai transform which maps Lagrangian submanifolds of $X$ to coherent sheaves on $\hat{X}$.
\end{itemize}
\end{conj}

\subsection{The D-brane effective action and the deformed Hermitian Yang-Mills equation}
D-branes play an important role in SYZ picture of mirror symmetry and homological mirror symmetry, so it is interesting to further study their behavior under mirror symmetry.
D-brane dynamics can be studied using the low energy effective action,  and in particular BPS solutions are described as critical points of this action. The bosonic part of supersymmetric Dirac-Born-Infeld (DBI) action of a single D$p$ brane has the following form:
\[
I_p=I_{DBI}+I_{WZ}=-T_p\int_W d^{p+1}\sigma \sqrt{g_{\mu\nu}+{\mathcal F}_{\mu\nu}}+\mu \int_W C\wedge e^{{\mathcal F}}.
\]
Here $g_{\mu\nu}$ is the pull-back of the metric, and ${\mathcal F}_{\mu\nu}$ is the modified two form ${\mathcal F}_{\mu\nu}=2\pi \alpha^{'}(F-B)$ with $F$ is the field strength of a gauge field on the D-brane world volume,
and $B$ the the pull-back of NS two form, often called the $B$-field. $T_p$ and $\mu$ are the brane tension and brane charges respectively, while $C$ is the formal sum of RR fields $C^{(r)}$:
\[
C=\sum_{r=0}^{10} C^{(r)}.
\]
This action is invariant under the $\kappa$ symmetry (a fermionic local symmetry) and space-time supersymmetry, and we can combine both in determining the fraction of unbroken supersymmetry by solving following equation
\[
(1-\Gamma)\eta =0.
\]
Here $\eta$ is the spacetime spinor, and $\Gamma$ is a Hermitian traceless matrix satisfying
\[
\text{tr} \Gamma =0, ~~\Gamma^2=1.
\]
Let's focus on a Calabi-Yau three manifold from now on, and assume we have a no-where vanishing holomorphic three form $\Omega$ and a Kahler form $J$ (to be consistent with the physics notation).
 The solutions of the BPS equation for the DBI action of a D$p$ brane were derived in \cite{MMMS}.  We have
\begin{itemize}
\item $p+1=3$: M is special Lagrangian, and the modified field strength ${\mathcal F}_{\mu\nu}=0$. The Lagrangian condition is $J|_M=0$ and the special Lagrangian condition is 
\begin{equation*}
\text{Im}~e^{\sqrt{-1}\theta} \Omega|_M=0.
\end{equation*}
\item $p+1=2n$ is even: M is  holomorphic, and the modified field strength satisfies the following equation
\begin{equation}\label{eq: dHYMphys}
\begin{aligned}
& {\mathcal F}^{2,0}=0,  \\
& \frac{1}{n!}(f^*(J)+{\mathcal F})^{n}=e^{\sqrt{-1}\theta}{\sqrt{|J+{\mathcal F}|}\over \sqrt{|J|}} vol(M).
\end{aligned}
\end{equation}
Here $f^*(J)$ is the pull back of Kahler form on cycle $M$. This equation is called the {\em deformed Hermitian Yang-Mills equation}. Notice that there is a $\alpha^{'}$ factor in front of ${\mathcal F}$, 
so the leading order term of the second equation is then
\[
{\mathcal F}\wedge J^{n-1}=c J^{n},
\]
which is simply the Hermitian-Yang-Mills equation. The second equation can also be put in the following form
\[
\text{Im}\left(e^{-\sqrt{-1}\theta}(J+{\mathcal F})^n\right)=0.
\]
where $J$ is the Kahler form.

\end{itemize}

\subsection{The semi-flat limit of SYZ mirror symmetry}
Let's summarize the BPS solution of the DBI action: we either have a special Lagrangian (sLag) cycle with vanishing gauge field strength, or a holomorphic 
cycle with connection satisfying deformed Hermitian-Yang-Mills (dHYM) equation. Mirror symmetry exchanges D-branes wrapping sLag cycles and D-brane wrapping holomorphic cycles. 
Thus, under SYZ mirror picture, we should see the exchange of sLag branes with flat connections and holomorphic branes with dHYM connections. 

The SYZ picture of mirror symmetry uses D-branes wrapping on the whole $T^3$ fibre of $X$, and the dual is a brane wrapped on a single point of the dual torus fibre of the mirror manifold $\hat{X}$.
If we study BPS D-branes wrapping on a single point of the torus fibre and wrapping on the whole base $B$, after T-duality the mirror should be a D brane wrapping on the whole manifold $\hat{X}$.

It is in general difficult to study the full moduli space of D-branes, but it is possible to check the above picture by looking at the semi-flat limit of the SYZ torus fibration \cite{LYZ}.  We briefly recall the set-up for semi-flat mirror symmetry, but refer the reader to the beautiful paper of Hitchin \cite{Hit, Hit1} (see also \cite{L}).  Fix an affine manifold $D$, which we assume is a domain in $\mathbb{R}^{n}$ (for example, the fundamental domain of a torus).  Let $x^i$ denote coordinates on $D$, and let $y^i$ be coordinates on $TD$ induced by
\[
(y^1,\ldots,y^n) \longmapsto \sum_iy^{i}\frac{\del}{\del x^i}.
\]
The bundle $TD$ carries a natural complex structure making the coordinates  $z^i=x^i+\sqrt{-1}y^i$ holomorphic.  At the same time, the bundle $T^{*}D$ carries a natural symplectic structure by defining 
\[
\omega= \sum_{i} dx^i \wedge dy^i.
\]
Let $\pi: TD \rightarrow D$, and $\hat{\pi}:T^{*}D \rightarrow D$ be the projections.  Let $\phi: D \rightarrow \mathbb{R}$ be a smooth strictly convex function solving the Monge-Amp\`ere equation
\[
\det \left( \frac{\del^2 \phi}{\del x_i \del x_j}\right) =1.
\]
Pulling back $\phi$ by $\pi$ to the total space of $TD$ induces a Calabi-Yau metric, and hence a metric on $T^{*}D$.  By the $2$ of out $3$ rule for K\"ahler manifolds, this induces a complex structure on $T^{*}D$.  We can compactify this picture by taking dual lattices $\Lambda\subset TD$, and  $\Lambda^{*}\subset T^{*}D$, and passing to the quotient $X:= TD/\Lambda$, $\hat{X}:= T^{*}D/\Lambda^{*}$.  In this case $X, \hat{X}$ are mirror Calabi-Yau manifolds.  This is semi-flat mirror symmetry.  In local coordinates, the Ricci-flat Kahler metric and K\"ahler form on $X$ are
\[
\begin{aligned}
& g=\sum_{i,j}{\partial \phi\over \partial x^i \partial x^j}(dx^idx^j+dy^idy^j), \nonumber\\
& \omega ={\sqrt{-1}\over2}\sum_{i,j}{\partial \phi\over \partial x^i \partial x^j} dz^i\wedge d\bar{z}^j, \nonumber\\
& \Omega=dz^1\wedge\ldots \wedge dz^n.
\end{aligned}
\]
The SYZ mirror $\hat{X}$ is found by T-duality on the torus fibers.  Let $\tilde{y}^{i}$ denote coordinates on $T^{*}D$ dual to $y^{i}$.  Define coordinates $\tilde{x}$ by the Legendre transform of $\phi$,
\[
{\partial \tilde{x}^j\over \partial x^k}=\phi_{jk}.
\]
 The reader can calculate directly that  $\tilde{z}^j=\tilde{x}^j+\sqrt{-1}\tilde{y}^j$ define holomorphic coordinates on $\hat{X}$.  In this notation the geometric data for $\hat{X}$ is:
\[
\begin{aligned}
& \tilde{g}=\sum_{i,j}\phi^{ij}(d\tilde{x}^id\tilde{x}^j+d\tilde{y}^id\tilde{y}^j) \nonumber\\
& \tilde{\omega} ={\sqrt{-1}\over2}\sum_{i,j}\phi^{ij} d\tilde{z}^i\wedge d\bar{\tilde{z}}^j \nonumber\\
&\tilde{\Omega}=d\tilde{z}^1\wedge\ldots\wedge d\tilde{z}^n.
\end{aligned}
\]
We now consider a section of the fibration $\hat{\pi}:\hat{X} \rightarrow D$, $\sigma := \{x \mapsto \tilde{y}^i(x)\}$.  This section will be Lagrangian if
\[
\frac{\del \tilde{y}^{i}}{\del x^k} = \frac{\del \tilde{y}^k}{\del x^i}
\]
which implies that $\sigma$ can be (locally) written as the graph of a $1$-form $df: D \rightarrow \hat{X}$.  Now we impose the assumption that $\sigma$ is special Lagrangian; namely
\[
{\rm Im}\left(e^{-\sqrt{-1}\hat{\theta}}\hat{\Omega}\right)\bigg|_{\sigma}=0
\]
for a constant $\hat{\theta}$.  It is most convenient to write the graph in terms of the Legendre transformed coordinates.  We have
\[
\frac{\del f}{\del x_j} dx^{j} = \phi^{j\ell} \frac{\del f}{\del x_j} d\tilde{x}^{\ell}
\]
so in terms of the Legendre transform coordinates the graph is
\[
x \mapsto \left(\tilde{x}^i = \tilde{x}^i(x),\quad  \tilde{y}^\ell= \phi^{j\ell} \frac{\del f}{\del x_j}\right).
\]
It follows that
\[
d\tilde{z}_i\bigg|_{\sigma}= \left[\phi_{ij} +\sqrt{-1} \left( \phi^{pi} \frac{\del f}{\del x_j \del x_p} - \phi^{pm}\phi_{jmk}\phi^{ki}\frac{\del f}{\del x_p}\right)\right]dx^j
\]
and thus the special Lagrangian condition is 
\[
{\rm Im}\left[e^{-\sqrt{-1}\hat{\theta}} \det \left(\phi_{ij} +\sqrt{-1} \left( \phi^{pi} \frac{\del f}{\del x_j \del x_p} - \phi^{pm}\phi_{jmk}\phi^{ki}\frac{\del f}{\del x_p}\right) \right)\right]=0
\]
In order to translate this to the mirror manifold $X$, we will need the Fourier-Mukai transform.  Fix a point $x \in D$, and consider the fiber $\hat{T} = \hat{\pi}^{-1}(x) \subset \hat{X}$.  This is the dual torus to $T = \pi^{-1}(x)\subset X$.  A point $\hat{y} \in \hat{T}$ defines a map $T \mapsto \mathbb{R}/\mathbb{Z}$, by $y\mapsto \hat{y}^jy_j$.  This map is induced from integrating the flat connection
\[
D_{A}:= d + \sqrt{-1}\tilde{y}^jdy_j
\]
on the trivial $\mathbb{C}$ bundle over $T$.  This construction, performed on each fiber, yields a $U(1)$ connection on $X$, and so a complex line bundle $L$ with connection $D_A$.  The curvature of this bundle is
\[
D_{A}^2 = \sqrt{-1}\sum_{i,j} \frac{\del \tilde{y}^{j}}{\del x_{i}} dx^i \wedge dy^j.
\]
The $(0,2)$ part of the curvature is given by 
\[
\frac{\del \tilde{y}^{j}}{\del x_{i}}  - \frac{\del \tilde{y}^{i}}{\del x_{j}} 
\]
and so the induced bundle $L$ has a holomorphic structure precisely when the section $\sigma$ is Lagrangian.  The complex structure of $L$ is given by the operator
\[
\dbar_{A} = \dbar  -\frac{1}{2}\tilde{y}^j d\bar{z}^{j}.
\]
Recall that $\sigma$ is the graph of $df$.  Therefore a holomorphic frame for $L$ is given by the section $\sigma = e^{f}$, and in this frame, the connection is the Chern connection with respect to the metric
\[
h = e^{2f}.
\]
Let's see what this corresponds to under the Fourier-Mukai transform.  The $(1,1)$ component of the curvature of the mirror line bundle $L$ with connection $D_A$ is
\[
\begin{aligned}
F_{i\bar{j}}dz^i\wedge d\bar{z}^j &= -\frac{1}{2}\left(\frac{\del \tilde{y}^i}{\del x_j} + \frac{\del \tilde{y}^j}{\del x_i}\right)\sqrt{-1}(dy^i\wedge dx^j - dx^i\wedge dy^j)\\
&= \frac{\del \tilde{y}^i}{\del x_j}\sqrt{-1}(x^j\wedge dy^i+ dx^i\wedge dy^j)
\end{aligned}
\]
where in the last line we used the Lagrangian condition.  Now, using the Legendre transform we can write
\[
\frac{\del \tilde{y}^i}{\del x_j} = \left( \phi^{pi} \frac{\del f}{\del x_j \del x_p} - \phi^{pm}\phi_{jmk}\phi^{ki}\frac{\del f}{\del x_p}\right)
\]
and so the special Lagrangian equation is equivalent to
\[
{\rm Im}\left(e^{-i\hat{\theta}}(\omega + F_A)^n\right)=0
\]
Summarizing we have that the curvature $F_A$ satisfies the following equations:
\[
\begin{aligned}
 F_A^{2,0}&=0 \nonumber\\
 {\rm Im }(\omega+F_A)^n&=\tan(\theta)\, {\rm Re} (\omega+F_A)^n.
\end{aligned}
\]
which is precisely the dHYM equation. This correspondence easily extends to the general setting where $\sigma$ is equipped with a flat $U(1)$ connection.

\section{Analytic aspects of the dHYM equation}

Let $(X,\omega)$ be a compact K\"ahler manifold, and let $\mathfrak{a} \in H^{1,1}(X,\mathbb{R})$ be a given cohomology class.  Often we will assume that $\fa = c_1(L)$ for some holomorphic line bundle $L$, but this is only for aesthetic purposes.  We do not assume $X$ is Calabi-Yau, as in general $BPS$ D-branes correspond to solutions of the deformed Hermitian-Yang-Mills (dHYM) equation supported on proper submanifolds of a Calabi-Yau.  We are interested in the following question.
\begin{que}
When does there exist a smooth representative $\alpha$ of the fixed class $\fa$ so that
\begin{equation}\label{eq: dHYM}
(\omega+\sqrt{-1}\alpha)^{n} = r e^{\sqrt{-1}\hat{\theta}} \omega^{n}
\end{equation}
where $e^{\sqrt{-1}\hat{\theta}} \in S^{1}$ is a constant, and $r: X \rightarrow \mathbb{R}_{>0}$ is a smooth function.
\end{que}

Strictly speaking, comparing the expression~\eqref{eq: dHYM} with~\eqref{eq: dHYMphys}, the reader will see that we are considering the dHYM equation for $L^{-1}$, but this is just a matter of convention. We make a few preliminary observations.  First, fix a point $p\in X$, and choose holomorphic coordinates centered at $p$ so that
\[
\omega(p) = \frac{\sqrt{-1}}{2}\sum_i dz_i\wedge d\bar{z}_i \qquad \alpha(p) = \frac{\sqrt{-1}}{2}\sum_{i}\lambda_i dz_i \wedge d\bar{z}_i.
\]
Invariantly, the numbers $\lambda_i$ are the eigenvalues of the relative endomorphism $\omega^{-1}\alpha$; we will sometimes refer to these as the eigenvalues of $\alpha$, and we hope that no confusion will result.  At $p$ we have
\[
\frac{(\omega+\sqrt{-1}\alpha)^{n}}{\omega^{n}}(p) = \prod_{i} (1+\sqrt{-1}\lambda_i) = r_{\omega}(\alpha)e^{\sqrt{-1} \Theta_{\omega}(\alpha)}
\]
where 
\begin{equation}\label{eq: defnOfOps}
r_{\omega}(\alpha) = \sqrt{ \prod_i (1+\lambda_i^2)}, \qquad \Theta_{\omega}(\alpha) = \sum_{i} \arctan (\lambda_i).
\end{equation}
In this notation the deformed Hermitian-Yang-Mills equation can be written has
\begin{equation}\label{eq: cxsLag}
\Theta_{\omega}(\alpha) = \hat{\theta} \qquad \mod 2\pi.
\end{equation}
The constant $e^{\sqrt{-1}\hat{\theta}}$ is determined by cohomology by the requirement
\[
\int_{X}(\omega+\sqrt{-1}\alpha)^{n} \in \mathbb{R}_{>0}e^{\sqrt{-1}\hat{\theta}}.
\]
From this observation we obtain the first obstruction to existence of solutions to the deformed Hermitian-Yang-Mills equation.
\begin{lem}
If there exists a solution to the deformed Hermitian-Yang-Mills equation then
\[
\int_{X}(\omega +\sqrt{-1} \alpha)^{n} \in \mathbb{C}^{*}.
\]
\end{lem}
This obstruction is non-trivial in dimensions $n \geq 3$, and we will return to it in the next section.  Fix a reference metric $\alpha_0 \in \fa$.  By the $\ddbar$-lemma, any representative of $\fa$ can be written as 
\[
\alpha_{\phi}: = \alpha_0 + \ddb \phi
\]
where $\phi: X \rightarrow \mathbb{R}$. By~\eqref{eq: cxsLag}, the deformed Hermitian-Yang-Mills equation is the natural complex geometric analog of the special Lagrangian graph equation, which we essentially recounted in Section~\ref{sec: phys}.  Let us recall this problem explicitly.  Let $\mathbb{C}^{n} = \mathbb{R}^{n} +\sqrt{-1} \mathbb{R}^{n}$, which we equip with the standard Calabi-Yau structure
\[
\omega =\frac{ \sqrt{-1}}{2} \sum_i dz_i \wedge d\bar{z}_i \qquad \Omega= dz_1\wedge dz_2 \wedge \ldots \wedge dz_n.
\]
Let $f: \mathbb{R}^{n} \rightarrow \mathbb{R}$, and consider the graph of the gradient map of $x\mapsto (x,\nabla f(x))$, which we denote by $L$.  We seek $f$ so that $L$ is {\em special Lagrangian} with respect to the Calabi-Yau structure defined by $\omega, \Omega$.  That is,
\[
\omega|_{L} =0\qquad \Omega|_{L} = e^{\sqrt{-1}\hat{\theta}} dVol_{L}
\]
for some constant $e^{\sqrt{-1}\hat{\theta}} \in S^{1}$.  A straightforward computation shows that this is equivalent to
\[
\sum_{i=1}^{n} \arctan(\lambda_i) = \hat{\theta} \quad \mod 2\pi
\]
where $\lambda_i$ are the eigenvalues of the $D^{2}f$.  Special Lagrangian manifolds were first introduced by Harvey-Lawson \cite{HL} as an example of a calibrated submanifold.  In particular, special Lagrangian submanifolds are automatically volume minimizing in their homology class. We refer the reader to \cite{Hit} for a beautiful introduction to study of sLag manifolds.

Solutions of the deformed Hermitian-Yang-Mills equation also minimize a certain volume functional.  Consider the map
\[
\fa \ni \alpha \longmapsto V_{\omega}(\alpha) := \int_{X} r_{\omega}(\alpha) \omega^{n}.
\]
 Where $r_{\omega}(\alpha)$ is defined in~\eqref{eq: defnOfOps}.  We have
 \begin{prop}[Jacob-Yau \cite{JY}]\label{prop: BPS}
 Define $\hat{r} \geq 0$ by
 \[
 \hat{r} = \left| \int_{X} (\omega+\sqrt{-1}\alpha)^{n} \right|
 \]
Then we have $ V_{\omega}(\alpha) \geq \hat{r}$.  Furthermore, a smooth form $\alpha$ minimizes $V_{\omega}(\cdot)$ if and only of $\alpha$ solves the deformed Hermitian-Yang-Mills equation.  In this case, the minimum value of $V_{\omega}$ is precisely $\hat{r} >0$.
\end{prop}
 
Note that since since $\arctan(\cdot): \mathbb{R} \rightarrow \mathbb{R}$ is increasing, $\Theta_{\omega}(\cdot)$ is an elliptic second order operator.  A consequence of this is
\begin{lem}[Jacob-Yau \cite{JY}]
Solutions of the deformed Hermitian-Yang-Mills equation are unique, up to addition of a constant.
\end{lem}
\begin{proof}
Suppose we have functions $\phi_i:X\rightarrow  \mathbb{R}$ where $i=1,2$, such that $\alpha_i := \alpha_{\phi_i}$ satisfy
\[
\theta_{\omega}(\alpha_i) = \theta_i
\]
for constants $\theta_i$.  Then $\phi_1 = \phi_2+c$ for some constant $c\in \mathbb{R}$.  Consider the function $\phi_1 -\phi_2$.  Let $p\in X$ be a point where $\phi_1-\phi_2$ achieves its infimum.  Then we have
\[
\alpha_1 \geq \alpha_2
\]
and hence $\theta_{\omega}(\alpha_1) \geq \theta_{\omega}(\alpha_2)$.  It follows that $\theta_1 \geq \theta_2$.  Swapping $1 \leftrightarrow 2$ we get that $\theta_1=\theta_2$.  Finally, we write
\[
0 = \int_{0}^{1} \ddt \Theta_{\omega}((1-t)\alpha_1 +t\alpha_2) = \left(\int_{0}^{1}L_{t}^{i\bar{j}} dt\right) \del_i \del_{\bar{j}} (\phi_1-\phi_2)
\]
where $L_{t}^{i\bar{j}}$ is the linearized operator of $\Theta_{\omega}(\cdot)$ at the point $(1-t)\alpha_1 +t\alpha_2$.  Since this is uniformly elliptic, the strong maximum principle implies $\phi_1-\phi_2$ is constant.
\end{proof}

A slightly more general result is

\begin{lem}\label{lem: liftAng}
Suppose $\omega$ is a K\"ahler form, and $\alpha \in \fa$ has the property that ${\rm osc}_{X}\Theta_{\omega}(\alpha) <\pi$.  Then
\begin{enumerate}
\item $\int_{X}(\omega+\sqrt{-1}\alpha)^{n} \in \mathbb{C}^{*}$.
\item Let $\theta_{\alpha} \in (-n\frac{\pi}{2}, n\frac{\pi}{2})$ be defined by
\[
\int_{X}(\omega+\sqrt{-1}\alpha)^{n} \in \mathbb{R}_{>0}e^{\sqrt{-1}\theta_{\alpha}} \qquad \theta_{\alpha} \in [\inf_{X}\theta_{\omega}(\alpha), \sup_{X} \theta_{\omega}(\alpha)].
\]
If  $\alpha'$ is another representative of the class $\fa$ with ${\rm osc}_{X}\Theta_{\omega}(\alpha') <\pi$, then we have $\theta_{\alpha} = \theta_{\alpha'}$.
\end{enumerate}
\end{lem}
\begin{proof}
The assumption that ${\rm osc}_{X}\theta_{\omega}(\alpha) <\pi$ implies that 
\[
\frac{(\omega+\sqrt{-1}\alpha)^n}{\omega^n}
\]
lies in a half space, and hence the integral cannot vanish.  To prove the second point define the interval
\[
I(\alpha) := [\inf_{X}\Theta_{\omega}(\alpha), \sup_{X} \Theta_{\omega}(\alpha)]. 
\]
Writing $\alpha' = \alpha+\ddb \phi$ and looking at the maximum and minimum of $\phi$ we see that $I(\alpha)\cap I(\alpha') \ne \emptyset$.  On the other hand, we have points  $\theta_{\alpha} \in I(\alpha)$ and $\theta_{\alpha'} \in I(\alpha')$ with $\theta_{\alpha} = \hat{\theta} \mod 2\pi = \theta_{\alpha'}$.  Since $I(\alpha), I(\alpha')$ have length $\pi$, this implies $\theta_{\alpha}=\theta_{\alpha'}$.
\end{proof}

\begin{defn}
Supposing that there exists some $\alpha \in \fa$ with ${\rm Osc}_{X}\Theta_{\omega}(\alpha) <\pi$, we will define $\theta = \theta_{\alpha}$ as in Lemma~\ref{lem: liftAng} to be the {\em lifted angle}.  Since this is independent of the choice of $\alpha$, we will drop the subscript $\alpha$.
\end{defn}

\begin{rk}
We note that the lifted angle is, a priori, not determined by cohomology.  We will discuss this issue in the next section.
\end{rk}

Let us now return to the problem of solving the deformed Hermitian-Yang-Mills equation.  Jacob-Yau \cite{JY} studied the solvability of the deformed Hermitian-Yang-Mills equation via a heat flow method.  They considered the flow
\begin{equation}\label{eq: heatFlow}
\ddt \phi = \Theta_{\omega}(\alpha_{\phi}) - \theta,
\end{equation}
where $\theta$ is the lifted angle (assuming this is well-defined).  They proved
\begin{thm}[Jacob-Yau \cite{JY}]
Suppose that $(X,\omega)$ has non-negative orthogonal bisectional curvature.  Let $L\rightarrow X$ be an ample line bundle. Let $h_0$ be a positively curved metric on $L$. Then for $k$ sufficiently large the heat flow~\eqref{eq: heatFlow} for metrics on $L^{k}$ with initial data $h_0^k$ exists for all time and converges to a solution of the deformed Hermitian-Yang-Mills equation.
\end{thm}

\begin{rk}
The reader can easily check that if $\fa$ is a K\"ahler class ample then for $k$ sufficiently large the lifted angle of $k\fa$ is well-defined.
\end{rk}

Furthermore, in dimension $2$, Jacob-Yau showed that the dHYM equation could be rewritten as the complex Monge-Amp\`ere equation.  As a result, on complex surfaces they gave necessary and sufficient algebraic conditions for the existence of solutions to the dHYM equation based Yau's solution of the complex Monge-Amp\`ere equation \cite{Y} and the Demailly-P\u{a}un characterization of the K\"ahler cone \cite{DP}.  In general it is desirable to obtain existence results for solutions of dHYM without any assumptions on the curvature of $(X,\omega)$.  Observe that if a solution $\alpha$ of the deformed Hermitian-Yang-Mills equation exists then for every $1\leq j \leq n$ we have
\[
\theta- \frac{\pi}{2} < \sum_{i\ne j} \arctan(\lambda_i) < \theta+\frac{\pi}{2},
\]
where $\lambda_i$ are the eigenvalues of $\alpha$.  Conversely, we have the following;

\begin{thm}[Collins-Jacob-Yau \cite{CJY}]
Suppose there exists a $(1,1)$ form $\chi \in \fa$ such that
\begin{equation}\label{eq: trivCond}
\Theta_{\omega}(\chi) \in ((n-2)\frac{\pi}{2}, n\frac{\pi}{2}).
\end{equation}
Let $\theta \in((n-2)\frac{\pi}{2}, n\frac{\pi}{2})$ be the lifted angle.  Suppose that for every $1 \leq j \leq n$ we have
\begin{equation}\label{eq: subSol}
\sum_{i\ne j} \arctan(\mu_i) \geq \theta - \frac{\pi}{2}.
\end{equation}
where $\mu_i$ are the eigenvalues of $\chi$.  Then there exists a smooth solution of the deformed Hermitian-Yang-Mills equation.
\end{thm}

We make a few remarks about the theorem.  First of all, the conditions are clearly necessary in order to solve the equation.  Secondly, the assumption that $\theta_{\omega}(\chi) \in ((n-2)\frac{\pi}{2}, n\frac{\pi}{2})$ is superfluous as soon as the lifted angle $\theta$ satisfies 
\[
\theta \geq  (n-2 +\frac{2}{n}) \frac{\pi}{2}.
\]
We remark also that if $\alpha$ is a K\"ahler form, then for $k$ sufficiently large we can always ensure that $k\alpha$ satisfies~\eqref{eq: trivCond}.  

\section{Algebraic aspects of the deformed Hermitian-Yang-Mills equation}

We now turn our attention to the algebraic aspects of the dHYM equation.  There are essentially two questions we would like to discuss in this section. 
\begin{enumerate}
\item Is it possible to define the lifted angle {\em algebraically}?
\item Are the algebraic obstructions to the existence of solutions to the deformed Hermitian-Yang-Mills equation?
\end{enumerate}

In regards to the second point, it is useful to recall the origin of the dHYM equation as the equation of motion for BPS $D$-branes on the B-model.  Douglas proposed a notion of $\Pi$-stability which he predicted would be related to the existence of BPS D-branes in mirror symmetry \cite{Doug, Doug1}.  Motivated by these ideas, Bridgeland \cite{Br} developed a theory of stability conditions on triangulated categories; we refer the reader to \cite{Asp} for a nice introduction to these ideas with connections to physics and mirror symmetry.  Since the dHYM equation is the geometric equation of motion for a BPS D-brane on the B-model, it is reasonable to expect that the solvability of the equation should be linked with $\Pi$-stability, or more generally Bridgeland stability.   The study of Bridgeland stability conditions has attracted considerable interest since their introduction.  Even a partial recounting of theory of Bridgeland stability conditions, and the many important results in this area, is far beyond the scope of this article.  Nevertheless, we will recall briefly the salient features which seem to appear in the study of dHYM; we refer the reader to \cite{MS} and the references therein for more on this active area of research.

We will focus specifically on the case of interest to mirror symmetry, so that the triangulated category is $D^{b}Coh(X)$.
\begin{defn}
A {\em slicing} $\cP$ of $D^{b}Coh(X)$ is a collection of subcategories $\cP(\varphi) \subset D^{b}Coh(X)$ for all $\varphi \in \mathbb{R}$ such that
\begin{enumerate}
\item $\cP(\varphi)[1] = \cP(\varphi+1)$ where $[1]$ denotes the ``shift" functor,
\item if $\varphi_1 > \varphi_2$ and $A\in \cP(\varphi_1)$, $B \in \cP(\varphi_2)$, then ${\rm Hom}(A,B) =0$,
\item every $E\in D^{b}Coh(X)$ admits a Harder-Narasimhan filtration by objects in $\cP(\phi_i)$ for some $1 \leq i \leq m$.
\end{enumerate}
\end{defn}

We refer to \cite{Br} for a precise definition of the Harder-Narasimhan property.  A Bridgeland stability condition on $D^{b}Coh(X)$ consists of a slicing together with a {\em central charge} (see below).  For BPS $D$-branes in the B-model, the relevant central charge is given by
\[
D^{b}Coh(X) \ni E \longmapsto Z_{\omega}(E):= -\int_{X}e^{-\sqrt{-1}\omega}ch(E).
\]
Often a factor of $\sqrt{Td(X)}$ is also included, but we will take the above choice (see, for example, \cite{BMT, AB}).
\begin{defn}
A Bridgeland stability condition on $D^{b}Coh(X)$ with central charge $Z_{\omega}$ is a slicing $\cP$ satisfying the following properties
\begin{enumerate}
\item For any non-zero $E\in \cP(\varphi)$ we have
\[
Z_{\omega}(E) \in \mathbb{R}_{>0} e^{\sqrt{-1}\varphi},
\]
\item
\[
C := \inf \left\{ \frac{|Z_{\omega}(E)|}{\|ch(E)\|} : 0 \ne E \in \cP(\varphi), \varphi \in \mathbb{R} \right\} >0
\]
where $\| \cdot \|$ is any norm on the finite dimensional vector space $H^{even}(X, \mathbb{R})$.
\end{enumerate}
\end{defn}

Given a Bridgeland stability condition we define $\mathcal{A} := \cP((0,1])$ which is called the {\em heart}.  An object $A \in \mathcal{A}$ is semistable (resp. stable) if, for every surjection $A\twoheadrightarrow B$ we have
\[
 \varphi(A) \leq (\text { resp.} <)\,\,\varphi(B).  
\]

In order to make aesthetic contact with Bridgeland stability we will consider throughout this section the case when $\fa = c_1(L)$ for some holomorphic line bundle $L$.  This does not serve any purpose other than to make the formulae slightly more appealing.  Furthermore, the dHYM equation with transcendental cohomology class also appears in mirror symmetry as the equation satisfied by  ``complexified K\"ahler forms" \cite{L}.  First we note that for representative $\alpha \in c_1(L)$ we have
\[
(\omega+\sqrt{-1}\alpha)^{n} = n!(\sqrt{-1})^n\left[e^{-\sqrt{-1}(\omega+\sqrt{-1}\alpha)}\right]_{top}
\]
and hence we have
\[
\int_{X}(\omega+\sqrt{-1}\alpha)^{n} = n!(\sqrt{-1})^n\int_{X}e^{-\sqrt{-1}(\omega)}ch(L).
\]
We are therefore lead to consider 
\[
Z_{\omega}(L) := -\int_{X}e^{-\sqrt{-1}\omega}ch(L).
\]
Note that if $L$ admits a solution of the deformed Hermitian-Yang-Mills equation with $\theta \in ((n-2)\frac{\pi}{2}, n\frac{\pi}{2})$ then ${\rm Im}(Z_{\omega}(L))>0$.  Define a path $\gamma(t) : [1,\infty) \rightarrow \mathbb{C}$ by
\[
\gamma(t) := Z_{t\omega}(L)= -\int_{X}e^{-t\sqrt{-1}\omega}ch(L).
\]
 If $\gamma(t) \in \mathbb{C}^{*}$, then we can define
\[
\theta(L) := \text{ Winding angle } \gamma(t)
\]
as $t$ runs from $+\infty$ to $1$.  In complex dimension $1$ we have
\[
\gamma(t) = -\int_{X}(c_1(L) - \sqrt{-1}\omega t) = \sqrt{-1}\int_{X}(t\omega+\sqrt{-1}c_1(L))
\]
and so 
\[
\theta(L) = {\rm Arg}_{p.v.}\int_{X}(\omega+\sqrt{-1}c_1(L)) + \frac{\pi}{2},
\]
where ${\rm Arg}_{p.v.}$ denotes the principal value of ${\rm Arg}$ with values in $(-\pi, \pi]$.  In dimension $2$ we have
\[
\gamma(t) = \frac{1}{2} \int_{X}t^{2}\omega^{2} - c_1(L)^{2} + \sqrt{-1}t\int_{X}c_1(L)\wedge \omega.
\]
If $\gamma(t) =0$ for some $t \in [1,\infty)$, then we must have
\[
\int_{X}c_1(L)\wedge \omega =0.
\]
But in this case the Hodge index theorem says that $\int_{X}c_1(L)^2 \leq 0$, and hence ${\rm Re}(\gamma(t)) \ne 0$.  Thus $\gamma(t)$ lies in $\mathbb{C}^{*}$ and hence $\theta(L)$ is well defined.  Furthermore, we have
\[
\theta(L)  = {\rm Arg}_{p.v.} \int_{X}(\omega+\sqrt{-1}c_1(L))^2
\]
In three dimensions we encounter the first difficulty.  We write
\[
\gamma(t)= \left(\int_{X}t^2\frac{c_1(L)\wedge\omega^2}{2} - \frac{c_1(L)^3}{6}\right) + \sqrt{-1}\left(\int_{X} t\frac{c_1(L)^2\wedge \omega}{2} - t^3\frac{\omega^{3}}{6}\right).
\]
In general, $\gamma(t)$ may pass through $0\in \mathbb{C}$, and in fact, one can construct examples of such behavior on the blow up of $\mathbb{P}^{3}$ in a point.  However, assuming we have a solution of the deformed Hermitian-Yang-Mills equation, we can prove that this is not the case.

\begin{prop}\label{prop: CNI}
Suppose $\alpha \in c_1(L)$ solves $\theta_{\omega}(\alpha) = \theta$ with $\theta \in(\frac{\pi}{2}, \frac{3\pi}{2})$.  Then $\gamma(t) \in \mathbb{C}^{*}$ for all $t\in[1,\infty)$.  This follows from the Chern number inequality
\[
\left(\int_{X} \omega^3 \right)\left(\int_{X} ch_3(L)\right) < 3 \left(\int_{X}ch_2(L)\wedge\omega\right)\left(\int_{X}ch_1(L)\wedge\omega^2\right)
\]
\end{prop}
\begin{proof}
We will use the deformed Hermitian-Yang-Mills equation pointwise to prove the inequality.  Suppose first that $\theta \in (\pi, \frac{3\pi}{2})$.  Since 
\[
\Theta_{\omega}(\alpha) = \sum_{i=1}^{3} \arctan(\lambda_i) =  \theta > \pi
\]
we must have that $\alpha$ is a K\"ahler form.  Since $c_1(L)$ admits a solution of the dHYM equation, if $\gamma(t)$ passes through the origin at time $T$, we must have that $T>1$.  Solving for $T$ we have
\[
 \left(\int_{X}T^2\frac{c_1(L)\wedge\omega^2}{2} - \frac{c_1(L)^3}{6}\right) =0,
\]
and so
\[
1 < T^2 = \frac{\int_{X}c_1(L)^3}{3\int_{X}c_1(L)\wedge\omega^2}.
\]
Plugging this into the equation for ${\rm Im}(\gamma(T))=0$ we see that we must have
\[
\left(\int_{X} \omega^3\right)\left(\int_{X}c_1(L)^3\right) = 9\left( \int_{X} c_1(L)^2\wedge \omega\right) \left(\int_{X}c_1 \wedge \omega^2\right)
\]
We will show this cannot happen.  Fix a point $p\in X$, and left $\lambda = (\lambda_1, \lambda_2, \lambda_3)$ be the eigenvalues of $\alpha$ with respect to $\omega$.  We write the deformed Hermitian-Yang-Mills equation as
\[
\tan(\theta)\left( \omega^3-3\alpha^2\wedge\omega\right) =  3\alpha\wedge \omega^2-\alpha^3.
\]
Let $\sigma_1, \sigma_2, \sigma_3$ be the symmetric functions of degree $1$, $2$, and $3$ respectively.  For example;
\[
\sigma_2(\lambda) = \lambda_1\lambda_2+ \lambda_2\lambda_3+\lambda_1\lambda_3.
\]
We have
\[
\alpha^{3} = \sigma_3(\lambda) \omega^3, \quad \alpha^{2}\wedge \omega = \sigma_2(\lambda)\frac{\omega^{3}}{3}, \quad \alpha\wedge \omega^2 = \sigma_1(\lambda)\frac{\omega^{3}}{3}
\]
and so we can write the deformed Hermitian-Yang-Mills equation as
\[
\tan(\theta)(1-\sigma_2)  = \sigma_1-\sigma_3
\]
Since $\lambda_i >0$ for all $i$ we have
\[
\sigma_1 + \tan(\theta)(\sigma_2-1) < \sigma_1\sigma_2.
\]
Since $\theta \in (\frac{\pi}{2}, 3\frac{\pi}{2})$ we have $1-\sigma_2 <0$, and so we obtain
\[
\tan(\theta)<\sigma_1.
\]
Since $\theta$ is constant we integrate both sides with respect to $\omega^{3}$ to get
\[
\tan(\theta) \int_{X}\omega^{3} < 3\int_{X}\alpha \wedge \omega^2.
\]
On the other hand, by definition we have
\[
\tan(\theta) =\frac{\int_{X} \alpha^3-3\alpha\wedge \omega^2}{ \int_{X}  3\alpha^2\wedge\omega-\omega^3}.
\]
By the assumption on $\theta$ the denominator is positive, and so we can rearrange this inequality to obtain the result.  The remaining case, when $\theta \in (\frac{\pi}{2}, \pi]$ is even easier, using only that
\[
\tan(\theta) \leq 0 < \sigma_1.
\]
We leave the details to the reader.
\end{proof}

With this proposition in hand it is easy to see that $\theta(L)$ is precisely the constant appearing on the right hand side of the deformed Hermitian-Yang-Mills equation, provided a solution exists.  The main new difficulty in dimension $3$ which is not present in dimension $1$ or $2$ is to determine the algebraically the lifted angle of solutions to the deformed Hermitian-Yang-Mills equation when $Z_{\omega}(L)$ has
\[
{\rm Re}(Z_{\omega}(L)) <0, {\rm Im}(Z_{\omega}(L)) >0.
\]
The primary difficult is that solutions to dHYM with phase $\theta \in(-\frac{3\pi}{2}, -\pi] \cup [\pi, 3\frac{\pi}{2})$ are both mapped into this quadrant.  One way to distinguish these two cases is to determine whether ${\rm Re}(Z_{t\omega}(L))$ is positive or negative when ${\rm Im}(Z_{t\omega}(L)) =0$.  This is precisely what the Chern number inequality proved in Proposition~\ref{prop: CNI} accomplishes.  In arbitrary dimension this problem will be even more complicated as it will require keeping track of the signs of the real and imaginary parts of $Z_{t\omega}(L))$ and any point time where $Z_{t\omega}(L)) $ crosses the real or imaginary axes.

We note that conjectural Chern number inequalities involving $ch_3$ have appeared in the literature on Bridgeland stability conditions  \cite{BMT}.  These inequalities play a fundamental role in establishing the existence of stability conditions.  We note, however, that a counter example to the conjectural inequality in \cite{BMT} was found by Schmidt \cite{Sch}.  It would be very interesting to extend these inequalities to higher rank bundles admitting solutions of dHYM.  We end by remarking that, in this correspondence between dHYM and Bridgeland stability, the lifted angle $\theta(L)$ is not the same as the slicing angle $\varphi(L)$; instead, the two are related by a constant depending on the dimension of support of $L$.  When $L$ is a line bundle this is nothing but the dimension of $X$, but similar ideas hold for line bundles supported on proper analytic sets, which appear as torsion sheaves in $D^{b}Coh(X)$. 

We now turn to the problem of finding algebro-geometric obstructions to the existence of solutions to the deformed Hermitian-Yang-Mills equation.  Recall that, if we have a solution of the deformed Hermitian-Yang-Mills equation with lifted phase $\theta \in (n-2\frac{\pi}{2}, n\frac{\pi}{2})$, then necessarily there is an element $\chi \in c_1(L)$ such that for each $1\leq j \leq n$ we have
\[
(n-1)\frac{\pi}{2} > \sum_{i\ne j} \arctan(\mu_i) > \theta - \frac{\pi}{2}.
\]
In fact, for every subset $J \subset \{1, 2,\ldots, n \}$ with $\#J=p$ we have
\[
(n-p)\frac{\pi}{2}>\sum_{i\notin J} \arctan(\mu_i) > \theta - p\frac{\pi}{2}.
\]
Consider the form
\[
(\omega+\sqrt{-1}\chi)^{n-1}.
\]
Fix a point $p\in X$, and choose coordinates so that
\[
\omega(p) = \frac{\sqrt{-1}}{2}\sum_i dz_i\wedge d\bar{z}_i \qquad \chi(p) = \frac{\sqrt{-1}}{2}\sum_{i}\mu_i dz_i \wedge d\bar{z}_i.
\]
we have
\[
(\omega+\sqrt{-1}\chi)^{n-1} = \sum_j r_j e^{\sqrt{-1} \sum_{i\ne j} \arctan(\mu_j)} \widehat{dz_j \wedge d\bar{z}_j}
\]
where $r_j >0$, and
\[
\widehat{dz_j \wedge d\bar{z}_j} = (\sqrt{-1})^{n-1} dz_1 \wedge d\bar{z}_1 \cdots\widehat{ dz_j \wedge d\bar{z}_j} \cdots dz_n \wedge d\bar{z}_n.
\]
Consider the real $(n-1,n-1)$ form given by
\[
\begin{aligned}
{\rm Im}\left( e^{-\sqrt{-1}\left(\theta -\frac{\pi}{2}\right)}(\omega+\sqrt{-1}\chi)^{n-1}\right) &=\sum_j r_j  {\rm Im}\left( e^{\sqrt{-1}\left(\sum_{i\ne j} \arctan(\mu_j) -(\theta -\frac{\pi}{2})\right)}\right) \widehat{dz_j \wedge d\bar{z}_j}
\end{aligned}
\]
By assumption we have for each $1\leq j \leq n$
\[
0< \sum_{i\ne j} \arctan(\mu_j) -(\theta -\frac{\pi}{2})< \pi
\]
and so
\[
{\rm Im}\left( e^{-\sqrt{-1}\left(\theta -\frac{\pi}{2}\right)}(\omega+\sqrt{-1}\chi)^{n-1}\right) >0
\]
in the sense of $(n-1,n-1)$ forms.  In particular, if $V \subset X$ is a irreducible analytic subvariety with $\dim_{\mathbb{C}} V=n-1$, then we must have
\[
{\rm Im}\left(\int_{V} e^{-\sqrt{-1}\left(\theta -\frac{\pi}{2}\right)}(\omega+\sqrt{-1}\chi)^{n-1}\right) >0.
\]
Similar inequalities hold for irreducible analytic subvarieties of all dimension, with the same proof. 
\begin{prop}
Suppose $c_1(L)$ admits a solution of the deformed Hermitian-Yang-Mills equation with $\theta \in ((n-2)\frac{\pi}{2}, n\frac{\pi}{2})$.  Then for every irreducible analytic subvariety $V\subset X$ of dimension $1\leq p <n$ we have
\begin{equation}\label{eq: obstr}
{\rm Im}\left(\int_{V} e^{-\sqrt{-1}\left(\theta -(n-p)\frac{\pi}{2}\right)}(\omega+\sqrt{-1}\chi)^{p}\right) >0.
\end{equation}
\end{prop}
In terms of the central charge we can write this in the following way.  Define
\[
Z_{\omega,V}(L) = -\int_{V}e^{-\sqrt{-1}\omega}ch(L)
\]
then we must have
\begin{equation}\label{eq: phaseIneq}
{\rm Im}\left(\frac{Z_{\omega, V}(L)}{Z_{\omega, X}(L)}\right) >0.
\end{equation}
To relate this to the Bridgeland stability condition we would like to think of inequality~\eqref{eq: phaseIneq} as saying that the surjection
\[
L \twoheadrightarrow L \otimes \mathcal{O}_{V}
\]
does not destabilize $L$, where $\mathcal{O}_{V}$ is the skyscraper sheaf with support on $V$.  Unfortunately this is not quite correct (unless $Td(X)=1$), since 
\[
Z_{\omega,V}(L) \ne Z_{\omega, X}(L\otimes \mathcal{O}_{V}).
\]   
Finally we note that if $L$ admits a solution of the deformed Hermitian-Yang-Mills equation then by the BPS bound in Proposition~\ref{prop: BPS} we have
\[
\frac{ |Z_{\omega, X}(L)|}{\|ch(L)\|} >0
\]
which is precisely the second condition required in the definition of a Bridgeland stability condition.

\begin{conj}[Collins-Jacob-Yau \cite{CJY}]\label{conj: CJY}
There exists a solution to the deformed Hermitian-Yang-Mills equation in the class $\fa$ with lifted angle $\theta \in(n-2)\frac{\pi}{2}, n\frac{\pi}{2})$ if and only if~\eqref{eq: obstr} holds for all proper, irreducible analytic subvarieties $V\subsetneq X$ with $\dim_{\mathbb{C}}V= p$.
\end{conj}

We end by noting some evidence for the conjecture.  First of all, the conjecture holds in complex dimension 2 \cite{JY, CJY}.  In general, we consider an asymptotic version of the conjecture.  Suppose that $L$ is an ample line bundle.  We ask the following question

\begin{que}
When can we find a metric $h$ on $L$, such that the induced metric $h^k$ on $L^k$ solves the deformed Hermtian-Yang-Mills equation for $k \gg 1$?
\end{que}
Replacing $\alpha \mapsto k\alpha$ we see that, up to rescaling, the deformed Hermitian-Yang-Mills equation on $L^k$ is 
\[
c\alpha^{n} = n \alpha^{n-1}\wedge \omega + O(k^{-1})
\]
where the constant $c$ is determined by
\[
c = \frac{n \int_{X}\alpha^{n-1}\wedge \omega}{\int_{X}\alpha^{n}}.
\]
This equation has a long history in K\"ahler geometry.  It was discovered independently by Donaldson \cite{Don} and Chen \cite{Chen04}, and was studied from the analytic point of view by Weinkove \cite{W1,W2} and Song-Weinkove \cite{SW3}.  Let us consider the asymptotic version of Conjecture~\ref{conj: CJY}.  First, we observe that
\[
\begin{aligned}
ch(L^{k}) &= \sum_{p=1}^{n} k^{p}\frac{c_1(L)^p}{p!}\\
ch(\mathcal{O}_{V}) &= (-1)^{d-1}(d-1)![V] + \text{ {\em higher degree} }
\end{aligned}
\]
and therefore
\[
\begin{aligned}
\int_{X}e^{-\sqrt{-1}\omega}ch(L^{k}\otimes \mathcal{O}_{V}) &=(-1)^{d-1}\frac{k^{d}}{d} \int_{V}c_1(L)^d\\
&\quad -\sqrt{-1} (-1)^{d-1}k^{d-1}\int_{V}c_1(L)^{d-1}\wedge \omega + \text{ {\em lower order} }.
\end{aligned}
\]
Note that this agrees with the leading order term in the expansion of $Z_{V,\omega}(L^k)$.  We obtain
\[
{\rm Arg}_{p.v.}Z_{X,\omega}(L^k \otimes \mathcal{O}_{V}) = \arctan\left( - \frac{1}{k} \frac{d\int_{V}c_1(L)^{d-1}\wedge \omega}{ \int_{V}c_1(L)^d} + \text{ {\em lower order }} \right).
\]
The formal limit of Conjecture~\ref{conj: CJY} is therefore
\begin{conj}[Lejmi-Sz\'ekelyhidi \cite{LS}]\label{conj: LS}
Let $L$ be an ample line bundle.  There exists a K\"ahler metric in $c_1(L)$ solving the $J$-equation if and only if 
\[
 \frac{d\int_{V}c_1(L)^{d-1}\wedge \omega}{ \int_{V}c_1(L)^d} <  \frac{n\int_{X}c_1(L)^{n-1}\wedge \omega}{ \int_{X}c_1(L)^n}
 \]
 for all irreducible analytic subvarieties $V\subsetneq X$ with $\dim_{\mathbb{C}}V=d$.
 \end{conj}
In fact, this conjecture was discovered from a very different point of view than the one discuss here.  The work of Lejmi-Sz\'ekelyhidi \cite{LS} is based on an extension of $K$-stability, which plays an important role in the existence of constant scalar curvature K\"ahler metrics \cite{Don02, Don05}. We have
\begin{thm}[Collins-Sz\'ekelyhidi \cite{CS}]
Conjecture~\ref{conj: LS} is true when $X$ is toric.
\end{thm}
\bigskip
\noindent {\bf Acknowledgements}: The authors are grateful to Adam Jacob for some helpful comments on an early draft of this paper.

\end{document}